\documentclass[12pt,reqno]{amsart}

\usepackage{times}
\usepackage{graphicx}
\usepackage{romannum}

\usepackage{amsmath,amsthm,amsfonts,amscd,amssymb}
\usepackage{pst-node}
\usepackage{pst-plot}

\usepackage[all]{xy}

\usepackage{tikz, braids}
\usetikzlibrary{decorations.pathreplacing, arrows}

\newtheorem{theorem}{Theorem}[section]

\newtheorem{lemma}[theorem]{Lemma}
\newtheorem{remark}[theorem]{Remark}
\newtheorem{proposition}[theorem]{Proposition}

\numberwithin{theorem}{section}
\numberwithin{equation}{section}

\newcommand{\F}{\mathbb F}

\newcommand{\T}{{\mathcal{T}}}

\newcommand{\norm}[1]{\left\Vert#1\right\Vert}

\newcommand{\la}{\langle}
\newcommand{\ra}{\rangle}

\newcommand{\Comp}{\mathbb{C}}

\newcommand{\n}{\mathbb{N}}

\begin{document}

\title{Additivity violation of the regularized minimum output entropy}

\author {Beno\^\i{}t Collins and Sang-Gyun Youn}

\address{Beno\^\i{}t Collins, 
Department of Mathematics, Graduate School of Science,  
Kyoto University, Kyoto 606-8502, Japan.}
\email{collins@math.kyoto-u.ac.jp}

\address{Sang-Gyun Youn, 
Department of Mathematics Education, Seoul National University, 
Gwanak-ro 1, Gwanak-gu, Seoul, Republic of Korea 08826}
\email{s.youn@snu.ac.kr}

\maketitle

\begin{abstract}
The problem of additivity of the Minimum Output Entropy is of fundamental importance in Quantum Information Theory (QIT). 
It was solved by Hastings \cite{Ha09} in the one-shot case, by exhibiting a pair of random quantum channels. 
However, the initial motivation was arguably to understand regularized quantities and there 
was so far no way to solve additivity 
questions 
in the regularized case. The purpose of this paper is to give a solution to this problem.
Specifically, we exhibit a pair of quantum channels which unearths additivity violation of the regularized minimum output entropy. 
Unlike previously known results in the one-shot case, our construction is non-random, infinite dimensional and in the commuting-operator setup. 
The commuting-operator setup is equivalent to the tensor-product setup in the finite dimensional case
for this problem, 
but their difference in infinite dimensional setting has attracted substantial attention 
and legitimacy recently in QIT with the celebrated resolutions of Tsirelson's and Connes embedding problem \cite{JiNaViWrYu20}. 
Likewise, it is not clear that our approach works in the finite dimensional setup. 
Our strategy of proof relies on developing a variant 
of the Haagerup inequality optimized for a product of free groups. 
\end{abstract}

\section{Introduction}

A crucial problem in quantum information theory is the problem of additivity of Minimum Output Entropy (MOE), which asks whether it is possible to find two quantum 
channels $\Phi_1,\Phi_2$ such that 
$$H_{min}(\Phi_1\otimes\Phi_2) < H_{min}(\Phi_1) + H_{min}(\Phi_2).$$

This problem was stated by King-Ruskai in \cite{KiRu01} as a natural question in the study of quantum channels. 
Shor proved in 2004 \cite{Sho04a,Sho04b} that a positive answer to the above question 
is equivalent to super-additivity of the Holevo capacity, i.e. there exist quantum channels $\Phi_1,\Phi_2$ such that
$$\chi (\Phi_1\otimes\Phi_2) > \chi (\Phi_1) + \chi (\Phi_2).$$

Heuristically, super-additivity of the Holevo capacity implies that entanglement inputs can be used to increase the transmission rate of classical information. We refer to Section \ref{sec:preliminaries} for the definitions of the MOE $H_{min}$ and the Holevo capacity $\chi$.
This question attracted lots of attention, and it was eventually solved by Hastings in 2009 \cite{Ha09}, with preliminary substantial contributions by Hayden, Winter, Werner, see in particular \cite{HaWi08}.
Subsequently, the mathematical aspects of the proof have been clarified in various directions by \cite{AuSzWe11,FuKiMo10,BrHo10,BeCoNe16,Co18,CoFuZh15}.

All previously known examples of additivity violation of MOE rely on subtle random constructions. In particular, to date, no deterministic
construction of additivity violation has ever been given. For attempts and partial results in the 
direction of non-random techniques we refer to \cite{WeHo02,GrHoPa10,BrCoLeYo20}, etc. 

Note that the above results do not imply anything about the problem of the additivity of the regularized MOE 
(see Definition \ref{def-MOE} for details). Indeed, additivity violation is not known to pertain when the MOE is regularized. 
More precisely, the additivity question for the regularized MOE asks whether it is possible to find two quantum channels 
$\Phi_1,\Phi_2$ such that 
$$\overline H_{min}(\Phi_1\otimes\Phi_2) < \overline H_{min}(\Phi_1) + \overline H_{min}(\Phi_2).$$
where $\overline H_{min}$ stands for the regularized MOE. This question was raised in \cite{Fu14} and the affirmative answer to this implies superadditivity of classical capacity.

Very few results are known about regularized entropic quantities -- see for example \cite{Ki02} or \cite{BrCoLeYo20} for partial results. 
In this paper, we focus on the additivity question of the regularized minimum output entropy, and the tensor product channel will 
be understood as a composition of two quantum channels whose systems of Kraus operators are commuting 
(see Section \ref{subsec-setup} for details).

In (quantum) information theory, 
one key paradigm is to allow repeated uses of a given quantum channel. To do this, we have to analyze a physical system by 
separated subsystems. In view of quantum strategies for non-local games, there are two natural models to describe separated
subsystems. One is the {\it tensor-product model} and the other is {\it commuting-operator model}. This latter approach is the object of 
intense research, see for example \cite{PaTo15,DyPa16,ClLiSl17,Sl20,ClCoLiPa18}, 
culminating with the recent resolution \cite{JiNaViWrYu20} in the negative of the celebrated Connes Embedding problem whose origin dates back to \cite{Co76}. 
In our case, commuting 
systems of Kraus operators correspond to a commuting-operator model. We refer to Section \ref{subsec-setup} for details on this.

The main result of this paper is an explicit construction of a pair of quantum channels $\Phi_1$ and $\Phi_2$ 
which have commuting systems of Kraus operators and satisfy additivity violation of the regularized 
MOE. Specifically, our main result can be stated as follows:

\begin{theorem}
There exist systems of operators $\left \{E_i\right\}_{i=1}^m$ and $\left \{F_j\right\}_{j=1}^n$ in $B(H)$ such that
\begin{enumerate}
\item $E_iF_j=F_jE_i$ for all $1\leq i\leq m$ and $1\leq j\leq n$,
\item $\displaystyle \sum_{i=1}^m E_i^*E_i=\mathrm{Id}_H=\sum_{j=1}^n F_j^*F_j$,
\item $\Phi_1,\Phi_2:\T(H)\rightarrow \T(H)$ are quantum channels given by
\[\Phi_1(\rho)=\displaystyle \sum_{i=1}^m E_i\rho E_i^*~\mathrm{and}~\Phi_2(\rho)=\displaystyle \sum_{j=1}^n F_j\rho F_j^*,\]
\item $\overline{H}_{min}(\Phi_1\circ \Phi_2)<\overline{H}_{min}(\Phi_1)+\overline{H}_{min}(\Phi_2)$.
\end{enumerate}
\end{theorem}
Note that the above discussion for the regularized MOE makes sense since the given channels are generated by finitely many 
Kraus operators, and given commuting systems will be chosen as an infinite dimensional analogue of i.i.d. Haar distributed 
unitary matrices, which will be explained in Section \ref{subsec-setup} and Theorem \ref{thm:superadditive} in details. 
One of the biggest benefit from this shift in perspective is that the regularized minimum output entropy becomes computable, 
whereas for random unitary channels, computing such regularized quantities still seems to remain totally out of reach at this point.

One of the key ingredients is to extend the {\it Haagerup inequality} \cite{Ha78} to products of free groups (Proposition \ref{thm1}). This result itself is a crucial fact. Indeed, the Haagerup inequality has numerous applications in operator algebras, non-commutative harmonic analysis and geometric group theory \cite{Bo81,CaHa85,Jo89,La00,La02}.

This paper is organized as follows. After this introduction, Section \ref{sec2} gathers some preliminaries about entropic quantities, 
quantum channels and the infinite dimensional framework. Section \ref{sec3} contains the proof of a Haagerup-type inequality 
for products of free groups as well as estimates for the regularized Minimum Output Entropy of our main family of quantum channels. 
Section \ref{sec4} explains how we can obtain additivity violation of the regularized MOE in the commuting operator setup, and 
Section \ref{sec5} contains concluding remarks.

\emph{Acknowledgements}: BC. was supported by JSPS KAKENHI 17K18734 and 17H04823.
S-G. Youn was funded by Natural Sciences and Engineering Research Council of Canada and by the National Research Foundation of Korea (NRF) grant funded by the Korea government (MSIT) (No. 2020R1C1C1A01009681).
S-G. Youn acknowledges the hospitality of Kyoto University on the occasion of two visits during which this project was initiated and completed. Part of this work was also done during the conference MAQIT 2019, at which the authors acknowledge a fruitful working environment. 
Finally, both authors would like to thank Mike Brannan, Jason Crann and Hun Hee Lee for inspiring discussions on this paper.

\section{Preliminaries}\label{sec2}

\subsection{Minimum output entropy in infinite dimensional setting}\label{sec:preliminaries}

Let $V:H_A\rightarrow H_B\otimes H_E$ be an isometry. Then partial traces on $H_B$ and $H_E$ define the following 
completely positive trace preserving maps (aka \emph{quantum channels})
\begin{align}
\Phi:\T(H_A)\rightarrow \T(H_B),&~\rho\mapsto (\mathrm{id}\otimes \mathrm{tr})(V\rho V^*)\\
\Phi^c:\T(H_A)\rightarrow \T(H_E),&~\rho\mapsto (\mathrm{tr}\otimes \mathrm{id})(V\rho V^*)
\end{align}
where $\mathcal{T}(H)$ denotes the space of trace class operators on a Hilbert space $H$. The map $\Phi^c$ is called 
the {\it complementary channel} of $\Phi$. The tensor product channels $\Phi^{\otimes k}: \T(H_A^{\otimes k})\rightarrow \T(H_B^{\otimes k})$ are defined in the obvious way.
A (quantum) state in $H$ is a positive element of $\T(H)$ of trace $1$, and for a state $\rho$, its R\'enyi entropy for $p\in (1,\infty )$ is defined as
$$H^p(\rho )= \frac{1}{1-p}\log \left ( tr (\rho^p)\right ).$$
Its limit as $p\to 1^+$ is called {\it the von Neumann entropy}, and if $\lambda_1(\rho) \ge \lambda_2 (\rho)\geq \ldots $ are the eigenvalues
of $\rho$ (counted with multiplicity), then the von Neumann entropy is
$$H(\rho ) = -\sum_i \lambda_i (\rho) \log \lambda_i (\rho ).$$

The \emph{Holevo capacity} of a quantum channel is 
$$\chi (\Phi ) = \sup \left \{ H(\Phi (\sum_i \lambda_i\rho_i)-\sum_i \lambda_i H (\Phi (\rho_i)) \right\},$$
where the supremum
is taken over all probability distributions $(p_i)_i$ and all families of states $(\rho_i)_i$. 
It describes the amount of classical information that can be carried through a single use of a quantum channel. If repeated uses 
of a given quantum channel is allowed, the ultimate transmission rate of classical information is described by
\[C(\Phi)=\lim_{k\rightarrow \infty}\frac{1}{k}\chi(\Phi^{\otimes k}),\]
which is called the {\it classical capacity}.

For a quantum channel $\Phi$, \emph{the Minimum Output Entropy }(MOE) and the \emph{regularized MOE} are defined as
\begin{align}\label{def-MOE}
H_{min}(\Phi)&=\inf_{\xi}H(\Phi(|\xi\ra\la \xi |))~\mathrm{and}\\
\overline{H}_{min}(\Phi)&=\lim_{k\rightarrow \infty}\frac{1}{k}H_{min}(\Phi^{\otimes k}).
\end{align}
respectively, where the infimum runs over all unit vectors $\xi$ of $H_A$. If $\Phi$ has finitely many Kraus operators 
$\left \{E_1,E_2,\cdots,E_N\right\}$ satisfying $\Phi(\rho)=\displaystyle \sum_{i=1}^N E_i\rho E_i^*$
(e.g. if $H_E$ is finite dimensional),
then
\begin{align*}
H_{min}(\Phi)+\chi(\Phi)\leq \log(N)~~\mathrm{and}~~\overline{H}_{min}(\Phi)+C(\Phi)\leq \log(N).
\end{align*}
\begin{remark}
In Equation (\ref{def-MOE}), taking the minimum over all states instead of pure states does not modify the
quantity thanks to operator convexity of the function $x\log(x)$, see e.g. \cite{Se60,NaUm61}.
\end{remark}

As in the finite dimensional setting, the following Schmidt decomposition theorem tells us that $\Phi(|\xi\ra\la \xi |)$ and $\Phi^c(|\xi\ra\la \xi |)$  have the same eigenvalues for each pure state $|\xi\ra\la \xi|\in \T(H)$.

\begin{proposition}
Let $V:H_A\rightarrow H_B\otimes H_E$ be an isometry and $\xi\in H_A$ be a unit vector. If we suppose that $\Phi(|\xi\ra\la \xi|)$ has the spectral decomposition $\displaystyle \sum_i \lambda_i |e_i\ra\la e_i|$ with $\lambda_i>0$, where $(e_i)_{i\in I}$ is an orthonormal subset of $H_B$, then there exists an orthonormal subset $(f_i)_{i\in I}$ of $H_E$ satisfying
\begin{equation}
V|\xi\ra=\sum_i \sqrt{\lambda_i } |e_i\ra \otimes |f_i\ra~\mathrm{and~}\Phi^c(|\xi\ra\la \xi|)=\sum_i \lambda_i |f_i\ra\la f_i |.
\end{equation}
In particular, $H(\Phi(|\xi\ra\la \xi |))=H(\Phi^c(|\xi\ra\la \xi|))$ for each unit vector $\xi\in H_A$.
\end{proposition}

\begin{proof}
Since $(e_i)_i$ is an orthonormal basis of $H_B$, we can write $V |\xi\ra$ as $\displaystyle \sum_i |e_i\ra\otimes |\eta_i\ra$ for a family $(\eta_i)_i \subseteq H_E$. Moreover, the given spectral decomposition of $\Phi(|\xi\ra\la \xi|)$ tells us that $\la \eta_j | \eta_i\ra=\lambda_i \delta_{i,j}$, which is equivalent to that $\displaystyle (f_i)_i:=\left ( \lambda_i^{-\frac{1}{2}}\eta_i \right )_i $ is an orthonormal set. Then we have
\[V|\xi\ra= \sum_i \sqrt{\lambda_i} |e_i\ra\otimes |f_i\ra~\mathrm{and~}\Phi^c(|\xi\ra\la \xi|)=\sum_i \lambda_i |f_i\ra\la f_i |.\]
\end{proof}

\subsection{Commuting systems of Kraus operators}\label{subsec-setup}

Let $H$ be a Hilbert space and $(a_{ij})_{(i,j)\in I\times J}$ be a family of bounded operators in $B(H)$ satisfying $\displaystyle \sum_{i\in I}a_{i,j}^*a_{i,j}=\mathrm{Id}_H$ for each $j\in J$. We assume that $I$ is finite and $J$ is arbitary. 
Let us define a family of quantum channels $(\Phi_j)_{j\in J}$ by
\[\Phi_j : \mathcal{T}(H)\to \T(H),~X\mapsto \sum_{i\in I} a_{ij}Xa_{ij}^*.\]
Their complement channels are given by
\[\Phi_j^c: \T(H)\to M_{|I|}(\Comp),~X\mapsto \sum_{i,i'\in I}\mathrm{tr} (a_{ij}Xa_{i'j}^*)|i\ra\la i'|.\]

We say that $(\Phi_j)_{j\in J}$ is in {\it commuting-operator setup} if the given channels $\Phi_j$ have commuting systems of 
Kraus operators in the sense that $a_{ij}a_{i'j'}=a_{i'j'}a_{ij}$ for any $i,i'\in I$ and $j,j'\in J$ such that $j\neq j'$.

 An example of this is the {\it tensor-product setup} but it is not the only example. A property is that $\Phi_j$ and $\Phi_{j'}$ commute 
 and their products are again quantum channels.
If $J$ is finite and $\Phi_1,\cdots,\Phi_{|J|}$ are in commuting-operator setup, then it is  natural to 
ask whether the following additivity property holds when $F$ is one of $H_{min},\overline{H}_{min},\chi,C$:
\[F\left (\prod_{j\in J}\Phi_j\right )=\sum_{j\in J}F(\Phi_j).\]

In particular, in the case $|J|=2$, the product channel $\Phi_1\circ \Phi_2$ is called a {\it local map} in the context of \cite{CrKrLeTo19}.

Let us construct a non-trivial quantum channel within the commuting-operator setup from the view of abstract harmonic analysis 
and operator algebra. Let $\mathbb{F}_{\infty}$ be the free group whose generators are $g_1,g_2,\cdots$ and let us define 
unitary operators $U_i$ and $V_j$ on $\ell^2(\mathbb{F}_{\infty})$ by
\[(U_if)(x)=f(g_i^{-1}x)~\mathrm{and~}(V_j f)(x)=f(xg_j)\]
for any $f\in \ell^2(\F_{\infty})$, $x\in \mathbb{F}_{\infty}$ and $i,j\in \n$. Since $U_iV_j=V_jU_i$ for all $i,j\in \n$, we have the 
following quantum channels that have commuting systems of Kraus operators.
\begin{align*}
\Phi_{N,l}:\T(\ell^2(\F_{\infty}))\rightarrow\T(\ell^2(\F_{\infty})),&~\rho\mapsto \frac{1}{N}\sum_{i=1}^N U_i\rho U_i^*~\mathrm{and}\\
\Phi_{N,r}:\T(\ell^2(\F_{\infty}))\rightarrow\T(\ell^2(\F_{\infty})),&~\rho\mapsto \frac{1}{N}\sum_{j=1}^N V_j\rho V_j^*.
\end{align*}

Let $J\in B(\ell^2(\F_{\infty}))$ be a unitary given by
\[(Jf)(x)=f(x^{-1})\]
for any $f\in \ell^2(\F_{\infty})$ and $x\in \F_{\infty}$. Then, since $JU_iJ=V_i$ and $J^2=\mathrm{Id}$, the above channels 
$\Phi_{N,l}$ and $\Phi_{N,r}$ are equivalent in the sense that
\[\Phi_{N,r}(\rho)=J\Phi_{N,l}(J\rho J)J\]
for any $\rho\in \T(\ell^2(\F_{\infty}))$. In particular, $H_{min}(\Phi_{N,l}^{\otimes k})=H_{min}(\Phi_{N,r}^{\otimes k})$ for any $k\in \n$.

Also, in order to express $k$-fold tensor product quantum channels $\Phi_{N,l}^{\otimes k}$, we will use the following notation
\[U_{m}= U_{m_1}\otimes U_{m_2}\otimes \cdots \otimes U_{m_k}\in B(\ell^2(\F_{\infty}^k))\]
for any $m=(m_1,\cdots,m_k)\in I^k$, where $I=\left \{1,2,\cdots,N\right\}$.

\section{Generalized Haagerup inequality and regularized MOE}\label{sec3}

In this section, we prove that the {\it Haagerup inequality} extends naturally to $r$-products of free groups $\F_{\infty}^r$.
Then we explain how this generalization allows to prove lower bounds of the regularized minimum output entropies (MOE) 
for $\Phi_{N,l}$. Let us simply write $\Phi_{N,l}$ as $\Phi_N$ in this section.

\subsection{A generalized Haagerup inequality}

For $x$ in the free group $\mathbb{F}_{\infty}$, we call $|x|$ its reduced word length with respect to the canonical generators
and their inverses. 
We consider products of free groups $\mathbb{F}_{\infty}^r$ for any $r\in \n$. 
Let us use the following notations $E_j=\left \{x\in \F_{\infty}:|x|=j\right\}$ for any $j\in \n_0$ and 
$E_m=E_{m_1}\times E_{m_2}\times \cdots \times E_{m_r}\subseteq \F_{\infty}^r$ for any $m=(m_1,\cdots,m_r)\in \n_0^r$.

We view $\mathbb{F}_{\infty}^r$ as an orthonormal basis that generates the Hilbert space $\ell^2(\F_{\infty}^r)$.
As an algebraic vector space, $\mathbb{F}_{\infty}^r$ spans $\mathbb{C}[\mathbb{F}_{\infty}^r]$, on which we may define
the convolution $f*g$ and the pointwise product $f\cdot g$. For $A\subset \mathbb{F}_{\infty}^r$,
$\chi_{A}$ denotes the indicator function of $A$.
First of all, we can generalize Lemma 1.3 of \cite{Ha78} as follows:

\begin{lemma}\label{lem1}

Let $l,m,k\in \n_0^r$ and let $f,g$ be supported in $E_k$ and $E_l$ respectively. Then 
\begin{equation}\label{eq1}
\norm{(f*g)\cdot \chi_{E_m}}_{\ell^2(\F_{\infty}^r)}\leq \norm{f}_{\ell^2(\F_{\infty}^r)}\cdot \norm{g}_{\ell^2(\F_{\infty}^r)}
\end{equation}
if $|k_j-l_j|\leq m_j\leq k_j+l_j$ and $k_j+l_j-m_j$ is even for all $1\leq j\leq r$. Otherwise, we have $\norm{(f*g)\cdot \chi_{E_m}}_{\ell^2(\F_{\infty}^r)}=0$.
\end{lemma}

\begin{proof}

Let us suppose that $|k_j-l_j|\leq m_j \leq k_j+l_j$ and $k_j+l_j-m_j$ is even for all $1\leq j\leq r$. If not, it is not difficult to see that $(f*g)\chi_m=0$. Also, it is enough to suppose that $f,g$ are finitely supported since $(f,g)\mapsto (f*g)\cdot \chi_{E_m}$ is bilinear. 

Let us use the induction argument with respect to $r\in \n$. The first case $r=1$ follows from \cite[Lemma 1.3]{Ha78} and let us suppose that (\ref{eq1}) holds true for $\mathbb{F}_{\infty}^r$. Under the notation $m=(m_0,m')\in \n_0^{r+1}$, we have
\begin{align}
&\norm{ (f*g)\chi_{E_{m}} }_{\ell^2(\mathbb{F}_{\infty}^{r+1})}^2=\sum_{s \in E_{m} }\left | \sum_{t,u\in \mathbb{F}_{\infty}^{r+1}:tu=s}  f(t)g(u) \right |^2\\
&=\sum_{s_0\in E_{m_0}}\sum_{s' \in E_{m'} }\left | \sum_{t_0,u_0\in \mathbb{F}_{\infty}:t_0 u_0=s_0}\sum_{t',u'\in \mathbb{F}_{\infty}^{r}:t'u'=s'}  f(t_0,t')g(u_0,u') \right |^2.
\end{align}

First of all, if we suppose that $m_0=k_0+l_0$, then we have
\begin{align*}
&=\sum_{s_0\in E_{m_0}}\sum_{s' \in E_{m'} } \sum_{t_0,u_0\in \mathbb{F}_{\infty}:t_0 u_0=s_0} \left | \sum_{t',u'\in \mathbb{F}_{\infty}^{r}:t'u'=s'}  f(t_0,t')g(u_0,u') \right |^2
\end{align*}
since there is a unique choice of $t_0\in E_{k_0}$ and $u_0\in E_{l_0}$ satisfying $t_0u_0=s_0$. Let us define functions $f_{t_0}(t')=f(t_0,t')$ and $g_{u_0}(u')=g(u_0,u')$ on $\mathbb{F}_{\infty}^r$. Then the above is written as
\begin{align}
\norm{ (f*g)\chi_{E_{m}} }_{\ell^2(\mathbb{F}_{\infty}^{r+1})}^2 = \sum_{s_0\in E_{m_0}}  \sum_{t_0,u_0\in \mathbb{F}_{\infty}:t_0 u_0=s_0} \norm{(f_{t_0}*g_{t_0})\chi_{E_{m'}}}_{\ell^2(\mathbb{F}_{\infty}^r)}^2,
\end{align}
which is dominated by
\begin{align}
&\leq \sum_{s_0\in E_{m_0}}  \sum_{t_0,u_0\in \mathbb{F}_{\infty}:t_0 u_0=s_0} \norm{f_{t_0}}_{\ell^2(\mathbb{F}_{\infty}^r)}^2\norm{g_{u_0}}_{\ell^2(\mathbb{F}_{\infty}^r)}^2\\
&\leq \sum_{t_0\in E_{k_0}}\sum_{u_0\in E_{l_0}} \norm{f_{t_0}}_{\ell^2(\mathbb{F}_{\infty}^r)}^2\norm{g_{u_0}}_{\ell^2(\mathbb{F}_{\infty}^r)}^2=\norm{f }_{\ell^2(\mathbb{F}_{\infty}^{r+1})}^2\norm{g }_{\ell^2(\mathbb{F}_{\infty}^{r+1})}^2.
\end{align}
Here, the first inequality comes from the induction hypothesis. Furthemore, the same idea applies whenever $m_j=k_j+l_j$ for some $0\leq j\leq r$. Now, let us suppose that $m_j<k_j+l_j$ and put $q_j=\displaystyle \frac{k_j+l_j-m_j}{2}$ for all $0\leq j\leq r$. Also, denote by $q=(q_0,\cdots,q_{r})$ and define two functions $F$ and $G$ on $\mathbb{F}_{\infty}^{r+1}$ as follows:
\begin{align}
F(x)&=\left \{ \begin{array}{cc}\left ( \sum_{v\in E_q}\left | f(xv) \right |^2 \right )^{\frac{1}{2}}&\text{for any }x\in E_{k-q} \\ 0 & \text{otherwise}\end{array} \right .\\
G(y)&=\left \{ \begin{array}{cc}\left ( \sum_{v\in E_q}\left | g(v^{-1}y) \right |^2 \right )^{\frac{1}{2}}&\text{for any }y\in E_{l-q} \\ 0 & \text{otherwise}\end{array} \right .
\end{align}
Note that $F$ and $G$ are supported in $E_{k-q}$ and $E_{l-q}$ respectively with
\begin{align}
\norm{F}_{\ell^2(\mathbb{F}_{\infty}^{r+1})}^2&=\sum_{x\in E_{k-q}}\sum_{v\in E_q}|f(xv)|^2=\norm{f}_{\ell^2(\mathbb{F}_{\infty}^{r+1})}^2\text{ and}\\
\norm{G}_{\ell^2(\mathbb{F}_{\infty}^{r+1})}^2&=\sum_{v\in E_q} \sum_{y\in E_{l-q}} |g(v^{-1}y)|^2=\norm{g}_{\ell^2(\mathbb{F}_{\infty}^{r+1})}^2.
\end{align}

Then we can show that the convolution $F*G$ dominates $|f*g|$ on $E_m$. Indeed, for any $s\in E_m$, there exists a unique $(x,y)\in E_{k-q}\times E_{l-q}$ such that $s=xy$ and we have
\begin{align*}
&|(f*g)(s)|=\left | \sum_{t\in E_k,u\in E_l: tu=s}f(t)g(u) \right |=\left | \sum_{v\in E_q}f(xv)g(v^{-1}y) \right |\\
&\leq \left ( \sum_{v\in E_q} | f(xv)|^2 \right )^{\frac{1}{2}}\left (\sum_{v\in E_q}|g(v^{-1}y)|^2 \right )^{\frac{1}{2}}=F(x)G(y)=(F*G)(s)
\end{align*}

Finally, since $F$ and $G$ are supported in $E_{k-q}$ and $E_{l-q}$ respectively with $m=(k-q)+(l-q)$, we can conclude that
\begin{align*}
\norm{(f*g)\chi_{E_m}}_{\ell^2(\mathbb{F}_{\infty}^{r+1})}&\leq \norm{(F*G)\chi_{E_m}}_{\ell^2(\mathbb{F}_{\infty}^{r+1})}\\
&\leq \norm{F}_{\ell^2(\mathbb{F}_{\infty}^{r+1})}\norm{G}_{\ell^2(\mathbb{F}_{\infty}^{r+1})}= \norm{f}_{\ell^2(\mathbb{F}_{\infty}^{r+1})}\norm{g}_{\ell^2(\mathbb{F}_{\infty}^{r+1})}.
\end{align*}

\end{proof}

Then we can generalize the Haagerup inequality to products of free groups $\F_{\infty}^r$ as follows:

\begin{proposition}\label{thm1}
Let $n=(n_1,\cdots,n_r)\in \n_0^r$ and $f$ be supported on $E_n\subseteq \F_{\infty}^r$. Then
\[\left \| L_f\right \|\leq (n_1+1)\cdots (n_r+1)\norm{f}_{\ell^2(\F_{\infty}^r)},\]
where $L_f$ is the convolution operator on $\ell^2(\F_{\infty}^r)$ given by $g\mapsto f*g$.
\end{proposition}
\begin{proof}
By density arguments, we may assume that $f$ is finitely supported and it is enough to consider finitely supported functions to evaluate the norm of the associated convolution operator $L_f$. Let $g\in \ell^2(\F_{\infty}^r)$ be finitely supported and define $g_k=g\cdot \chi_{E_k}$ for each $k\in \n_0^r$. Then $g=\displaystyle \sum_{k\in \n_0^r}g\cdot \chi_{E_k}$ and we have
\[h:=f*g=\sum_{k\in \n_0^r}f*g_k.\]
Then, by Lemma \ref{lem1}, we have the following estimate for $h_m=h\cdot \chi_{E_m}$ with $m=(m_1,\cdots,m_r)\in \n_0^r$ as follows:
\begin{align*}
&\norm{h_m}_{\ell^2(\F_{\infty}^r)}=\norm{\sum_{k\in \n_0^r}(f*g_k)\cdot \chi_{E_m}}_{\ell^2(\F_{\infty}^r)}\\
&\leq \sum_{k\in \n_0^r}\norm{(f*g_k)\cdot \chi_{E_m}}_{\ell^2(\F_{\infty}^r)}\leq  \sum_{\substack{k\in \n_0^r\\ n_j+k_j-m_j:even\\ |n_j-k_j|\leq m_j\leq n_j+k_j}}\norm{f}_{\ell^2(\F_{\infty}^r)} \norm{g_k}_{\ell^2(\F_{\infty}^r)} =:A
\end{align*}
Writing $k_j=m_j+n_j-2l_j$ for all $1\leq j\leq r$, we obtain
\begin{align*}
A&=\norm{f}_{\ell^2(\F_{\infty}^r)}\sum_{\substack{l_1,\cdots,l_r\\ 0\leq l_j\leq \min\left \{m_j,n_j \right\}}}\norm{g_{m+n-2l}}_{\ell^2(\F_{\infty}^r)}\\
\leq& \norm{f}_{\ell^2(\F_{\infty}^r)} \sqrt{(1+n_1)\cdots (1+n_r)} \left ( \sum_{\substack{l_1,\cdots,l_r\\ 0\leq l_j\leq \min\left \{m_j,n_j \right\}}}\norm{g_{m+n-2l}}_{\ell^2(\F_{\infty}^r)}^2 \right )^{\frac{1}{2}}.
\end{align*}
Therefore,
\begin{align*}
&\norm{h}_{\ell^2(\F_{\infty}^r)}^2=\sum_{m\in \n_0^r}\norm{h_m}_{\ell^2(\F_{\infty}^r)}^2\\
&\leq (1+n_1)\cdots (1+n_r)\norm{f}_{\ell^2(\F_{\infty}^r)}^2\sum_{m\in \n_0^r}\sum_{\substack{l_1,\cdots,l_r\\ 0\leq l_j\leq \min\left \{m_j,n_j \right\}}}\norm{g_{m+n-2l}}_{\ell^2(\F_{\infty}^r)}^2\\
&= (1+n_1)\cdots (1+n_r)\norm{f}_{\ell^2(\F_{\infty}^r)}^2\sum_{\substack{l_1,\cdots,l_r\\ 0\leq l_j\leq n_j}}\sum_{\substack{m_1,\cdots,m_r\\ l_j\leq m_j<\infty}}\norm{g_{m+n-2l}}_{\ell^2(\F_{\infty}^r)}^2\\
&= (1+n_1)\cdots (1+n_r)\norm{f}_{\ell^2(\F_{\infty}^r)}^2\sum_{\substack{l_1,\cdots,l_r\\ 0\leq l_j\leq n_j}}\sum_{\substack{k_1,\cdots,k_r\\ n_j-l_j\leq k_j<\infty}}\norm{g_{k}}_{\ell^2(\F_{\infty}^r)}^2\\
&\leq  (1+n_1)\cdots (1+n_r)\norm{f}_{\ell^2(\F_{\infty}^r)}^2\sum_{\substack{l_1,\cdots,l_r\\ 0\leq l_j\leq n_j}}\sum_{k\in \n_0^r}\norm{g_{k}}_{\ell^2(\F_{\infty}^r)}^2\\
&=(1+n_1)^2\cdots (1+n_r)^2\norm{f}_{\ell^2(\F_{\infty}^r)}^2 \norm{g}_{\ell^2(\F_{\infty}^r)}^2,
\end{align*}
which gives us
\[\norm{L_f}\leq (1+n_1)\cdots (1+n_r)\norm{f}_{\ell^2(\F_{\infty}^r)}.\]
\end{proof}

\subsection{A  norm estimate}

In this subsection, we investigate  the operator norm of the following elements
$$\sum_{v,w\in E^k} a_{v,w}(U_{v})^*U_{w} ~\mathrm{with~}a_{vw}\in \Comp,$$ 
where $E=\left \{g_1,g_2,\cdots\right\}$ is the set of generators of $\mathbb{F}_{\infty}$ and $E^k=E\times \cdots \times E\subseteq E_{(1,\cdots,1)}$. 
Indeed, in Section \ref{sec-reg.MOE}, 
this estimate will be needed to evaluate the regularized MOE.
Our estimate is as follows:

\begin{theorem}\label{thm2} 
For any $a=(a_{vw})_{v,w\in E^k}\in M_{N^k}(\Comp)$ such that $\mathrm{tr}(a)=0$, we have
$$\left \| \sum_{v,w\in E^k} a_{vw}(U_{v})^*U_{w}\right \| \le N^{\frac{k}{2}}\sqrt{(1+9N^{-1})^k-1}\left \|a \right \|_2.$$
\end{theorem}
\begin{proof}
Now we are summing over $(N^k)^2$ elements and we have to split the sum according to whether there are simplification or not. To do this, for any subset $K$ of $\left \{1,2,\cdots,k\right\}$, we define
\[E_K=\left \{(v,w)\in E^k\times E^k:v_i=w_i~\mathrm{for~all~}i\in K\right\}.\]

Note that $\displaystyle E^2=\sqcup_{K\subset \left \{1,\cdots,k\right\}}E_K$. Then, by the triangle inequality, we have
\begin{align*}
\norm{\sum_{v,w\in E^2}a_{vw} (U_v)^* U_w }&\le \sum_{K\subset \left \{1,\cdots,k\right\}}\norm{\sum_{(v,w)\in E_K} a_{vw} (U_v)^* U_w }\\
&=\sum_{s=0}^k \sum_{\substack{K\subset \left \{1,\cdots,k\right\}\\ |K|=s }}\norm{\sum_{(v,w)\in E_K} a_{vw} (U_v)^* U_w }.
\end{align*}

Note that $|K|=k$ implies $\displaystyle \sum_{(v,w)\in E_K}a_{vw} (U_v)^*U_w  =\sum_{v\in E^k}a_{vv}=0$. From now on, let us suppose that $|K|=s<k$. Then $E_K$ can be identified with the set $\left \{(z,x,y)\in E^s\times E^{k-s}\times E^{k-s}:x_j\neq y_j~\forall 1\leq j\leq k-s\right\}$ for each $K$. Under this notation, $(a_{vw})$ can be written as $(a^{K,z}_{x,y})$ and we have
\begin{align*}
\sum_{\substack{K\subset \left \{1,\cdots,k\right\}\\ |K|=s }} \norm{\sum_{(v,w)\in E_K}a_{vw} (U_v)^*U_w }\leq \sum_{\substack{K\subset \left \{1,\cdots,k\right\}\\ |K|=s }}\sum_{z\in E^s}\norm{\sum_{\substack{(x,y)\in E^{k-s}\times E^{k-s}\\ x_j\neq y_j}}a^{K,z}_{x,y} (U_x)^*U_y}.
\end{align*}

Moreover, Proposition \ref{thm1} gives us
\begin{align*}
&\leq \sum_{\substack{K\subset \left \{1,\cdots,k\right\}\\ |K|=s }}\sum_{z\in E^s} 3^{k-s}\left (\sum_{\substack{(x,y)\in E^{k-s}\times E^{k-s}\\ x_j\neq y_j}}|a^{K,z}_{x,y}|^2\right )^{\frac{1}{2}}
\end{align*}
and, since we are summing $N^s {k \choose s}$ elements, the Cauchy-Schwarz inequality tells us that
\begin{align*}
&\leq  3^{k-s} N^{\frac{s}{2}}{k \choose s}^{\frac{1}{2}} \left (\sum_{\substack{K\subset \left \{1,\cdots,k\right\}\\ |K|=s }}\sum_{z\in E^s}\sum_{\substack{(x,y)\in E^{k-s}\times E^{k-s}\\ x_j\neq y_j}}|a^{K,z}_{x,y}|^2\right )^{\frac{1}{2}}\\
 &= 3^{k-s}N^{\frac{s}{2}}{k \choose s}^{\frac{1}{2}}\left (\sum_{\substack{K\subset \left \{1,\cdots,k\right\}\\ |K|=s }}\sum_{(u,v)\in E_K} |a_{vw}|^2\right )^{\frac{1}{2}}.
\end{align*}

Here, the Haagerup constant $3^{k-s}$ appears due to $s$ cancellations. Applying the Cauchy-Schwartz inequality once more, 
we obtain that 

\begin{align*}
\norm{\sum_{v,w\in E^2} a_{vw}(U_v)^* U_w }&=\sum_{s=0}^{k-1} \sum_{\substack{K\subset \left \{1,\cdots,k\right\}\\ |K|=s }}\norm{\sum_{(v,w)\in E_K} a_{vw}(U_v)^*U_w}\\
&\leq \sum_{s=0}^{k-1} 3^{k-s}N^{\frac{s}{2}}{k \choose s}^{\frac{1}{2}}\left (\sum_{\substack{K\subset \left \{1,\cdots,k\right\}\\ |K|=s }}\sum_{(u,v)\in E_K} |a_{vw}|^2\right )^{\frac{1}{2}}\\
&\leq \left (  \sum_{s=0}^{k-1} 9^{k-s}N^s {k \choose s}\right )^{\frac{1}{2}}\left (\sum_{ K\subset \left \{1,\cdots,k\right\}}\sum_{(u,v)\in E_K} |a_{vw}|^2\right )^{\frac{1}{2}}\\
&=\sqrt{(N+9)^k-N^k}\norm{a}_2\\
&=N^{\frac{k}{2}}\sqrt{(1+9N^{-1})^k-1}\norm{a}_2.
\end{align*}

\end{proof}

\subsection{The regularized minimum output entropy of $\Phi_N$}\label{sec-reg.MOE}

Theorem \ref{thm2} enables us to show that for any density matrix $S$
\begin{equation}
\norm{(\Phi_N^c)^{\otimes k}(S)-\frac{1}{N^k}\mathrm{Id_N}^{\otimes k}}_2
\end{equation}
is sufficiently small with respect to the Hilbert-Schmidt norm. This generalizes \cite[Theorem 3.1]{Co18}.
Specifically, we prove the following theorem:

\begin{theorem}\label{thm3}
For each $k\in \n$, we have
\begin{equation}
\sup_{S}\norm{(\Phi_N^c)^{\otimes k}(S)-\frac{1}{N^k}\mathrm{Id}^{\otimes k}_N }_{2}\leq \frac{\sqrt{(1+9N^{-1})^k-1}}{N^{\frac{k}{2}}},
\end{equation}
where $S$ runs over all density matrices in $\mathcal{T}(\ell^2(\F_{\infty}^k))$.
\end{theorem}

\begin{proof}
Let $\displaystyle X=(x_{i,j})_{i,j\in I^k}=(\Phi_N^c)^{\otimes k}(S)-\frac{1}{N^k}\mathrm{Id}^{\otimes k}_N$. Since $\mathrm{tr}(X)=0$, we have
\begin{align*}
\mathrm{tr}(X^2)&=\mathrm{tr}((\Phi_N^c)^{\otimes k}(S)X)=\mathrm{tr}(S \left ((\Phi_N^c)^{\otimes k}\right )^*\left (X\right ))
\end{align*}
where $\left ( (\Phi_N^{c})^{\otimes k}\right )^*$ denotes the adjoint map of $(\Phi_N^{c})^{\otimes k}$. Moreover,
\[\left ((\Phi_N^c)^{\otimes k}\right )^*\left (X\right )=\frac{1}{N^k}\sum_{\substack{i,j\in I^k\\ i\neq j}}x_{i,j}U_i^*U_j\]
where $I=\left \{1,2,\cdots, N\right\}$. Hence, we have
\begin{align*}
\mathrm{tr}(S ((\Phi_N^c)^{\otimes k})^*(X))&\leq \frac{1}{N^k}\norm{\sum_{\substack{i,j\in I^k\\ i\neq j}}x_{i,j}U_i^*U_j}.
\end{align*}

According to Theorem \ref{thm2}, we get
\begin{equation}
\norm{X}_2^2=\mathrm{tr}(X^2)\leq \frac{N^{\frac{k}{2}}\sqrt{(1+9N^{-1})^k-1}\norm{X}_2}{N^k},
\end{equation}
as claimed.
\end{proof}
This allows us to estimate  the regularized minimum output entropies of $\Phi_N$  as follows:

\begin{theorem}\label{thm4}
For any $k\in \n$ we have
\[ H_{min}(\Phi_N^{\otimes k})\geq k\log(N) -2\log(1+\sqrt{(1+9N^{-1})^k-1}). \]
In particular, we have the following estimate for the regularized MOE
\[\overline{H}_{min}(\Phi_N)\geq \log(N)-\log(1+\frac{9}{N})\geq \log(N)-\frac{9}{N}.\]
\end{theorem}
\begin{proof}
Thanks to the fact that the von Neumann entropy is bigger than R{\' e}nyi entropies
of order $\alpha=2$ \cite{MuDuSzFeTo13}, we have
\begin{align*}
H((\Phi_N^c)^{\otimes k}(S))&\geq \frac{\alpha}{1-\alpha}\log \left (\left \| (\Phi_N^c)^{\otimes k}(S) \right \|_{\alpha}\right )\\
&\geq  -2\log \left ( N^{-\frac{k}{2}}(1+\sqrt{(1+9N^{-1})^k-1}) \right )\\
&=k\log(N)-2\log \left ( 1+\sqrt{(1+9N^{-1})^k-1} \right )
\end{align*}
for any density matrix $S\in \T(\ell^2(\F_{\infty}^k))$ by Theorem \ref{thm3}. In particular, we have
\[H_{min}(\Phi_N^{\otimes k})\geq k\log(N) -2\log(1+\sqrt{(1+9N^{-1})^k-1})\]
and the last conclusion follow from the following computation with L'H\^opital's rule:
\begin{align*}
&\lim_{k\rightarrow \infty}\frac{\log(1+\sqrt{(1+9N^{-1})^k-1})}{k}\\
&=\lim_{k\rightarrow \infty}\frac{\frac{(1+\frac{9}{N})^k\log(1+\frac{9}{N})}{2\sqrt{(1+\frac{9}{N})^k-1}}}{1+\sqrt{(1+\frac{9}{N})^k-1}}=\frac{1}{2}\log(1+\frac{9}{N}).
\end{align*}
\end{proof}

\section{Additivity violation of the regularized MOE}\label{sec4}

In this section, we choose two copies of $\Phi_N$ as $\Phi_{N,l}$ and $\Phi_{N,r}$. Indeed, these two quantum channels $\Phi_{N,l}$ and $\Phi_{N,r}$ are equivalent as explained in Section \ref{subsec-setup} and are in commuting-operator setup.

Then we can obtain the following additivity violation of the regularized MOE by generalizing Winter-Holevo-Hayden-Werner trick for $\Phi_{N,l}\circ \Phi_{N,r}$:

\begin{theorem}\label{thm:superadditive}
The regularized MOE is not additive: For any $N>e^{18}$, we have
\[\overline{H}_{min}(\Phi_{N,l}\circ \Phi_{N,r}) <\overline{H}_{min}(\Phi_{N,l})+\overline{H}_{min}(\Phi_{N,r}).\]
\end{theorem}

\begin{proof}

Note that under notations from Subsection \ref{subsec-setup},
\[(\Phi_{N,l}\circ \Phi_{N,r})(\rho)=\frac{1}{N^2}\sum_{i,j=1}^N U_iV_j \rho V_j^*U_i^*.\]
Since $|e\ra\la e|$ is an invariant for $U_iV_i$, we have
\begin{align*}
(\Phi_{N,l}\circ \Phi_{N,r})(|e\ra\la e|)=\frac{1}{N}|e\ra\la e|+\frac{1}{N^2}\sum_{i,j:i\neq j} |g_ig_j^{-1}\ra\la g_ig_j^{-1}| ,
\end{align*}
which implies 
\begin{align*}
\overline{H}_{min}(\Phi_{N,l}\circ \Phi_{N,r})&\leq H_{min}(\Phi_{N,l}\circ \Phi_{N,r})\leq H((\Phi_{N,l}\circ \Phi_{N,r})(|e\ra\la e|))\\
=& \frac{\log(N)}{N}+(N^2-N)\cdot \frac{\log(N^2)}{N^2}= 2\log(N)-\frac{\log(N)}{N}.
\end{align*}

Moreover, $\Phi_{N,l}$ and $\Phi_{N,r}$ are copies of $\Phi_N$, so that we have
\[2\log(N)-\frac{\log(N)}{N}<2\log(N)-\frac{18}{N}\leq \overline{H}_{min}(\Phi_{N,l})+\overline{H}_{min}(\Phi_{N,r})\] 
by Theorem \ref{thm4} if $N> e^{18}$.
\end{proof}

\section{Concluding remarks}\label{sec5}

(1) Various versions of $C^*$-tensor products can be used to obtain commuting systems of operators. For example, let $A,B$ be 
unital $C^*$-algebras and take families of operators $(E_i)_{i=1}^m\subseteq A$ and $(F_j)_{j=1}^n\subseteq B$. Also, suppose 
that $A\otimes_{max}B\subseteq B(K)$. Then 
\[\left \{E_i\otimes 1_B\right\}_{i=1}^m~\mathrm{and~}\left \{1_A\otimes F_j\right\}_{j=1}^n\]
give us commuting systems of operators in $B(K)$. Moreover, if we suppose that 
$C_r^*(\F_{\infty})\otimes_{max}C_r^*(\F_{\infty})\subseteq B(K)$ where $C_r^*(\F_{\infty})$ is the reduced group $C^*$-algebra of 
the free group $\F_{\infty}$, then commuting systems $\left \{U_i\otimes \mathrm{Id}\right\}_i$ and 
$\left \{\mathrm{Id}\otimes U_j\right\}_j$ give another example of additivity violation in the commuting-operator setup.

(2) Since Haagerup type inequalities exist for other groups (e.g hyperbolic groups \cite{Ha88}) or certain reduced free products of 
$C^*$-algebras \cite{Bo91}, it is natural to expect that similar results should hold and will yield other examples of additivity violation phenomena.

(3) It is worthwhile to compare the main results of this paper and the cases of random unitary channels. On the side of random 
unitary channels, the regularized MOE is unknown, whereas our Theorem \ref{thm4} gives us a strong estimate for the 
regularized MOE of $\Phi_N$. 

(4) One might wonder if we can evaluate the classical capacity of $\Phi_N^c$ whose output space is finite dimensional. 
Thanks to Theorem \ref{thm4} and a standard argument, the classical capacity of $\Phi_N^c$ is upper bounded by 
\[C(\Phi_N^c)\leq \log(N)-\overline{H}_{min}(\Phi_N^c)\leq \frac{9}{N}.\]

However, unlike in the tensor-product setup \cite[Theorem 6.1]{Fu14}, it is not clear whether additivity violation of the regularized MOE implies super-additivity of the classical capacity within commuting-operator framework,
so the question of the additivity of the classical capacity remains open.

\bibliographystyle{alpha}
\bibliography{Collins-Youn-polished}

\newcommand{\etalchar}[1]{$^{#1}$}
\def\cprime{$'$}
\begin{thebibliography}{MLDS{\etalchar{+}}13}

\bibitem[ASW11]{AuSzWe11}
Guillaume Aubrun, Stanis{\l}aw Szarek, and Elisabeth Werner.
\newblock Hastings's additivity counterexample via {D}voretzky's theorem.
\newblock {\em Comm. Math. Phys.}, 305(1):85--97, 2011.

\bibitem[BaH10]{BrHo10}
Fernando G. S.~L. Brand\~{a}o and Micha\l Horodecki.
\newblock On {H}astings' counterexamples to the minimum output entropy
  additivity conjecture.
\newblock {\em Open Syst. Inf. Dyn.}, 17(1):31--52, 2010.

\bibitem[BCLY20]{BrCoLeYo20}
Michael Brannan, Beno\^{\i}t Collins, Hun~Hee Lee, and Sang-Gyun Youn.
\newblock Temperley-{L}ieb {Q}uantum {C}hannels.
\newblock {\em Comm. Math. Phys.}, 376(2):795--839, 2020.

\bibitem[BCN16]{BeCoNe16}
Serban~T. Belinschi, Beno{\^{\i}}t Collins, and Ion Nechita.
\newblock Almost one bit violation for the additivity of the minimum output
  entropy.
\newblock {\em Comm. Math. Phys.}, 341(3):885--909, 2016.

\bibitem[Bo{\.z}81]{Bo81}
Marek Bo{\.z}ejko.
\newblock Remark on {H}erz-{S}chur multipliers on free groups.
\newblock {\em Math. Ann.}, 258(1):11--15, 1981.

\bibitem[Bo{\.{z}}91]{Bo91}
Marek Bo{\.{z}}ejko.
\newblock A {$q$}-deformed probability, {N}elson's inequality and central limit
  theorems.
\newblock In {\em Nonlinear fields: classical, random, semiclassical
  ({K}arpacz, 1991)}, pages 312--335. World Sci. Publ., River Edge, NJ, 1991.

\bibitem[CCLP18]{ClCoLiPa18}
Richard Cleve, Beno\^{\i}t Collins, Li~Liu, and Vern Paulsen.
\newblock Constant gap between conventional strategies and those based on
  {C$^*$}-dynamics for self-embezzlement.
\newblock {\em arXiv preprint arXiv:1811.12575}, 2018.

\bibitem[CFZ15]{CoFuZh15}
Beno\^{\i}t Collins, Motohisa Fukuda, and Ping Zhong.
\newblock Estimates for compression norms and additivity violation in quantum
  information.
\newblock {\em Internat. J. Math.}, 26(1):1550002, 20, 2015.

\bibitem[CKLT19]{CrKrLeTo19}
Jason Crann, David~W Kribs, Rupert~H Levene, and Ivan~G Todorov.
\newblock State convertibility in the von neumann algebra framework.
\newblock {\em Comm. Math. Phys., to appear, arXiv:1904.12664}, 2019.

\bibitem[CLS17]{ClLiSl17}
Richard Cleve, Li~Liu, and William Slofstra.
\newblock Perfect commuting-operator strategies for linear system games.
\newblock {\em J. Math. Phys.}, 58(1):012202, 7, 2017.

\bibitem[Col18]{Co18}
Beno\^{i}t Collins.
\newblock Haagerup's inequality and additivity violation of the minimum output
  entropy.
\newblock {\em Houston J. Math.}, 44(1):253--261, 2018.

\bibitem[Con76]{Co76}
A.~Connes.
\newblock Classification of injective factors. {C}ases {$II_{1},$} {$II_{\infty
  },$} {$III_{\lambda },$} {$\lambda \not=1$}.
\newblock {\em Ann. of Math. (2)}, 104(1):73--115, 1976.

\bibitem[DCH85]{CaHa85}
Jean De~Canni\`ere and Uffe Haagerup.
\newblock Multipliers of the {F}ourier algebras of some simple {L}ie groups and
  their discrete subgroups.
\newblock {\em Amer. J. Math.}, 107(2):455--500, 1985.

\bibitem[dlH88]{Ha88}
Pierre de~la Harpe.
\newblock Groupes hyperboliques, alg\`ebres d'op\'{e}rateurs et un
  th\'{e}or\`eme de {J}olissaint.
\newblock {\em C. R. Acad. Sci. Paris S\'{e}r. I Math.}, 307(14):771--774,
  1988.

\bibitem[DP16]{DyPa16}
Kenneth~J. Dykema and Vern Paulsen.
\newblock Synchronous correlation matrices and {C}onnes' embedding conjecture.
\newblock {\em J. Math. Phys.}, 57(1):015214, 12, 2016.

\bibitem[FKM10]{FuKiMo10}
Motohisa Fukuda, Christopher King, and David~K. Moser.
\newblock Comments on {H}astings' additivity counterexamples.
\newblock {\em Comm. Math. Phys.}, 296(1):111--143, 2010.

\bibitem[Fuk14]{Fu14}
Motohisa Fukuda.
\newblock Revisiting additivity violation of quantum channels.
\newblock {\em Comm. Math. Phys.}, 332(2):713--728, 2014.

\bibitem[GHP10]{GrHoPa10}
Andrzej Grudka, Micha\l Horodecki, and \L~ukasz Pankowski.
\newblock Constructive counterexamples to the additivity of the minimum output
  {R}\'{e}nyi entropy of quantum channels for all {$p>2$}.
\newblock {\em J. Phys. A}, 43(42):425304, 7, 2010.

\bibitem[Haa79]{Ha78}
Uffe Haagerup.
\newblock An example of a nonnuclear {$C\sp{\ast} $}-algebra, which has the
  metric approximation property.
\newblock {\em Invent. Math.}, 50(3):279--293, 1978/79.

\bibitem[Has09]{Ha09}
Matthew~B. Hastings.
\newblock Superadditivity of communication capacity using entangled inputs.
\newblock {\em Nature Physics}, 5:255--257, 2009.

\bibitem[HW08]{HaWi08}
Patrick Hayden and Andreas Winter.
\newblock Counterexamples to the maximal {$p$}-norm multiplicity conjecture for
  all {$p>1$}.
\newblock {\em Comm. Math. Phys.}, 284(1):263--280, 2008.

\bibitem[JNV{\etalchar{+}}20]{JiNaViWrYu20}
Zhengfeng Ji, Anand Natarajan, Thomas Vidick, John Wright, and Henry Yuen.
\newblock {MIP$^*$}= {RE}.
\newblock {\em arXiv preprint arXiv:2001.04383}, 2020.

\bibitem[Jol89]{Jo89}
Paul Jolissaint.
\newblock {$K$}-theory of reduced {$C^*$}-algebras and rapidly decreasing
  functions on groups.
\newblock {\em $K$-Theory}, 2(6):723--735, 1989.

\bibitem[Kin02]{Ki02}
Christopher King.
\newblock Additivity for unital qubit channels.
\newblock {\em J. Math. Phys.}, 43(10):4641--4653, 2002.

\bibitem[KR01]{KiRu01}
Christopher King and Mary~Beth Ruskai.
\newblock Minimal entropy of states emerging from noisy quantum channels.
\newblock {\em IEEE Trans. Inform. Theory}, 47(1):192--209, 2001.

\bibitem[Laf00]{La00}
Vincent Lafforgue.
\newblock A proof of property ({RD}) for cocompact lattices of {${\rm
  SL}(3,\bold R)$} and {${\rm SL}(3,\bold C)$}.
\newblock {\em J. Lie Theory}, 10(2):255--267, 2000.

\bibitem[Laf02]{La02}
Vincent Lafforgue.
\newblock {$K$}-th\'{e}orie bivariante pour les alg\`ebres de {B}anach et
  conjecture de {B}aum-{C}onnes.
\newblock {\em Invent. Math.}, 149(1):1--95, 2002.

\bibitem[MLDS{\etalchar{+}}13]{MuDuSzFeTo13}
Martin M\"{u}ller-Lennert, Fr\'{e}d\'{e}ric Dupuis, Oleg Szehr, Serge Fehr, and
  Marco Tomamichel.
\newblock On quantum {R}\'{e}nyi entropies: a new generalization and some
  properties.
\newblock {\em J. Math. Phys.}, 54(12):122203, 20, 2013.

\bibitem[NU61]{NaUm61}
Masahiro Nakamura and Hisaharu Umegaki.
\newblock A note on the entropy for operator algebras.
\newblock {\em Proc. Japan Acad.}, 37:149--154, 1961.

\bibitem[PT15]{PaTo15}
Vern~I. Paulsen and Ivan~G. Todorov.
\newblock Quantum chromatic numbers via operator systems.
\newblock {\em Q. J. Math.}, 66(2):677--692, 2015.

\bibitem[Seg60]{Se60}
I.~E. Segal.
\newblock A note on the concept of entropy.
\newblock {\em J. Math. Mech.}, 9:623--629, 1960.

\bibitem[Sho04a]{Sho04b}
Erratum: ``{E}quivalence of additivity questions in quantum information
  theory'' [{C}omm. {M}ath. {P}hys. {\bf 246} (2004), no. 3, 453--472;
  mr2053939] by {P}. {W}. {S}hor.
\newblock {\em Comm. Math. Phys.}, 246(3):473, 2004.

\bibitem[Sho04b]{Sho04a}
Peter~W. Shor.
\newblock Equivalence of additivity questions in quantum information theory.
\newblock {\em Comm. Math. Phys.}, 246(3):453--472, 2004.

\bibitem[Slo20]{Sl20}
William Slofstra.
\newblock Tsirelson's problem and an embedding theorem for groups arising from
  non-local games.
\newblock {\em J. Amer. Math. Soc.}, 33(1):1--56, 2020.

\bibitem[WH02]{WeHo02}
R.~F. Werner and A.~S. Holevo.
\newblock Counterexample to an additivity conjecture for output purity of
  quantum channels.
\newblock {\em J. Math. Phys.}, 43(9):4353--4357, 2002.
\newblock Quantum information theory.

\end{thebibliography}

\end{document}